\documentclass[12pt]{amsart}
\title[The weight and Lindel\"of property]{The weight and Lindel\"of property 
in spaces and topological groups}

\author[M. Tkachenko]{M. Tkachenko}
\address[M. Tkachenko]{\hfill\break
Departamento de Matem\'aticas\hfill\break Universidad Aut\'onoma
Metropolitana\hfill\break Av. San Rafael Atlixco 186,\hfill\break
Col. Vicentina, Del. Iztapalapa, C.P. 09340\hfill\break Mexico D.F.,
Mexico} \email{mich@xanum.uam.mx}

\keywords{Weight; Density; $\om$-narrow; Lindel\"of $\Sigma$-space; Separable; Countably compact}
\subjclass{Primary 54H11, 22A05, 54G10; Secondary 54D20, 54G20, 54C10,
54C45, 54D60}
\usepackage{amsthm} 
\usepackage{amsmath}
\usepackage{amssymb}
\usepackage{eucal}
\usepackage[curve,arrow,matrix,tips]{xy}

\DeclareMathOperator{\Inte}{Int}

\def\om{\omega}
\def\cont{\mathfrak c}
\def\sm{\setminus}

\def\emp{\emptyset}

\def\R{\mathbb R}
\def\T{\mathcal T}

\def\T{\mathbb T}
\def\phi{\varphi}

\let\sub=\subseteq
\let\res=\restriction

\newtheorem{thm}{Theorem}[section]
\newtheorem{coro}[thm]{Corollary}
\newtheorem{prop}[thm]{Proposition}
\newtheorem{lemma}[thm]{Lemma}

\newtheorem{problem}[thm]{Problem}
\newtheorem{example}[thm]{Example}

\theoremstyle{definition}

\newtheorem{Def}[thm]{Definition}
\theoremstyle{plain}

\begin{document}
\date{January 24, 2015}

\begin{abstract}
We show that if $Y$ is a dense subspace of a Tychonoff space $X$, then $w(X)\leq nw(Y)^{Nag(Y)}$, where $Nag(Y)$ is the \textit{Nagami number} of $Y$. In particular, if $Y$ is a Lindel\"of $\Sigma$-space, then $w(X)\leq nw(Y)^\omega\leq nw(X)^\omega$.

Better upper bounds for the weight of topological groups are given. For example, if a topological group $H$ contains a dense subgroup $G$ such that $G$ is a Lindel\"of $\Sigma$-space, then $w(H)=w(G)\leq \psi(G)^\om$. Further, if a Lindel\"of $\Sigma$-space $X$ generates a dense subgroup of a topological group $H$, then $w(H)\leq 2^{\psi(X)}$.

Several facts about subspaces of Hausdorff separable spaces are established. It is well known that the weight of a separable Hausdorff space $X$ can be as big as $2^{2^\cont}$. We prove on the one hand that if a \textit{regular} Lindel\"of $\Sigma$-space $Y$ is a subspace of a separable Hausdorff space, then $w(Y)\leq 2^\om$, and the same conclusion holds for a Lindel\"of $P$-space $Y$. On the other hand, we present an example of a countably compact topological group $G$ which is homeomorphic to a subspace of a separable Hausdorff space and satisfies $w(G)=2^{2^\cont}$, i.e.~has the maximal possible weight. 
\end{abstract}

\maketitle

\noindent \textit{MSC (2000):} \small{54H11, 54A25, 54C30}\medskip


\section{Introduction}\label{intro}
It is known that the number of continuous real-valued functions, $|C(X)|$, on a Tychonoff space $X$ is not defined by the weight of $X$, even if $w(X)=2^\om=\cont$ --- it suffices to take as $X_1$ a discrete space $D$ of cardinality $\cont$ and as $X_2$ the one-point compactification of $D$. Then the weights of $X_1$ and $X_2$ coincide, while $|C(X_1)|=2^\cont > \cont=|C(X_2)|$. In any case, the cardinality of $C(X)$ always satisfies $w(X)\leq |C(X)|\leq 2^{d(X)}$, where $d(X)$ is the density of the Tychonoff space $X$.

The upper bound for $|C(X)|$ in the latter inequality is not the best possible. It is shown by Comfort and Hager in \cite{CH} that \emph{every} space $X$ satisfies $|C(X)|\leq w(X)^{wl(X)}$, where $wl(X)$ is the \textit{weak Lindel\"of number} of $X$ (see Subsection~\ref{SubS} below), and that the inequality $w(X)^{wl(X)}\leq 2^{d(X)}$ holds for every \emph{regular} space $X$. In particular, if $X$ has countable cellularity or contains a dense Lindel\"of subspace, then $|C(X)|\leq w(X)^\om$. If in addition $X$ is Tychonoff and the weight of $X$ is equal to $\cont$, then clearly $|C(X)|=\cont$. Thus the number of continuous real-valued functions on a Tychonoff weakly Lindel\"of space $X$ is completely defined by the weight of $X$ provided that $w(X)=\cont$ or, more generally, $w(X)=\kappa^\om$ for an infinite cardinal $\kappa$.

 One of our principal results in Section~\ref{Sec:Con}, Theorem~\ref{Th:2}, states that if $Y$ is a dense subspace of a Tychonoff space $X$, then $|C(X)|\leq nw(Y)^{Nag(Y)}$, where $Nag(Y)$ is the \textit{Nagami number} of the space $Y$ (see Subsection~\ref{SubS}). Therefore, if $X$ is a regular Lindel\"of $\Sigma$-space, then $w(\beta{X})\leq nw(X)^\om$, where $\beta{X}$ is the Stone-\v{C}ech compactification of $X$. In particular, if a regular Lindel\"of $\Sigma$-space $X$ satisfies $nw(X)=\kappa^\om$ for some $\kappa\geq\om$, then $w(X)=nw(X)$ and $w(X)=w(\beta{X})=|C(X)|=\kappa^\om$. Therefore, the cardinality of $C(X)$ is completely defined by the weight of $X$ in this case.  
 
In Section~\ref{Sec:3} we consider topological groups that contain a dense subgroup or a subspace which is a Lindel\"of $\Sigma$-space. Again, our aim is to estimate the weight of the enveloping group in terms of cardinal characteristics of the corresponding dense subgroup or subspace. A typical result there is Theorem~\ref{Th:3} stating that if a Lindel\"of $\Sigma$-group $G$ is a dense subgroup of a topological group $H$, then $w(H)=w(G)\leq \psi(G)^\om$. Similarly, if a Lindel\"of $\Sigma$-space $X$ generates a dense subgroup of a topological group $H$, then $w(H)\leq 2^{\psi(X)}$ (Theorem~\ref{Th:3b}). 

To extend the aforementioned results to wider classes of topological groups we introduce the notion of \textit{$(\kappa,\lambda)$-moderate group}, where $\om\leq\kappa\leq\lambda$. As an application of the new concept we deduce in Corollary~\ref{Cor:4} that every Lindel\"of $\om$-stable topological group $G$ with $\psi(G)\leq\cont$ satisfies $w(G)\leq\cont$.

Our aim in Section~\ref{Sec:4} is to find out what kind of restrictions a \textit{Tychonoff} space $Y$ must satisfy in order that $Y$ be a subspace of a separable \textit{Hausdorff} space $X$. One of the obvious restrictions on $Y$ is the inequality $|Y|\leq |X|\leq 2^\cont$. A less trivial restriction is found in the recent article \cite{LMT}: \textit{If $Y$ is a compact subspace of a separable Hausdorff space $X$, then $w(Y)\leq\cont$.} It is worth noting that the weight of a separable Hausdorff space can be as big as $2^{2^\cont}$ \cite{JK}. Making use of Theorem~\ref{Th:1} we extend the result from \cite{LMT} to Lindel\"of $\Sigma$-spaces: \textit{If a regular Lindel\"of $\Sigma$-space $Y$ is homeomorphic to a subspace of a separable Hausdorff space, then $w(Y)\leq\cont$} (see Theorem~\ref{Th:4}). The same conclusion is valid if $Y$ is a Lindel\"of \textit{$P$-space}, i.e.~all $G_\delta$-sets in $Y$ are open (see Theorem~\ref{Th:5}).

It is established in \cite{LMT} that there are wide classes of topological groups $G$ with the following property: If $G$ is \textit{homeomorphic} to a subspace of a separable Hausdorff space, then $G$ itself is separable and, hence, satisfies $w(G)\leq\cont$. In particular, so is the class of \textit{almost connected pro-Lie groups} which includes connected locally compact groups and their finite or infinite products. In Proposition~\ref{Pro:Emb2} we show that this is not true anymore for countably compact topological groups. In fact, we prove that there exists a countably compact Abelian topological group $G$ homeomorphic to a subspace of a separable Hausdorff space such that $d(G)=2^\cont$ and $w(G)=2^{2^\cont}$, i.e.~$G$ has the maximal possible density and weight.

\subsection{Notation and terminology}\label{SubS}
All spaces considered here are assumed to be Tychonoff if the otherwise is not mentioned explicitly. The exception is Section~\ref{Sec:4}, where we consider Hausdorff spaces. 

By $w(X)$, $nw(X)$, $d(X)$, $l(X)$, $wl(X)$, and $c(X)$ we denote the weight, network weight, density, Lindel\"of number, weak Lindel\"of number, and cellularity of a given space $X$, respectively.
The character and pseudocharacter of $X$ are $\chi(X)$ and $\psi(X)$.

Let $\beta{X}$ be the Stone-\v{C}ech compactification of a Tychonoff space $X$. Denote by $\mathcal{C}$ the family of all closed subsets of $\beta{X}$. We say that a subfamily $\mathcal{F}$ of $\mathcal{C}$ \textit{separates points of $X$ from $\beta{X}\setminus X$} provided that for every pair of points $x\in X$ and $y\in\beta{X}\setminus X$, there exists $F\in\mathcal{F}$ such that $x\in F$ and $y\notin F$. Then we put 
$$
Nag(X)=\min\{|\mathcal{F}|: \mathcal{F}\subset\mathcal{C}\ \mbox{ and } \mathcal{F} \mbox{ separates points of } X \mbox{ from }  \beta{X}\setminus X\}.
$$
If $Nag(X)\leq\omega$, we say that $X$ is a \textit{Lindel\"of $\Sigma$-space} (see \cite[Section~5.3]{AT}). The class of Lindel\"of $\Sigma$-spaces is countably productive and is stable with respect to taking $F_\sigma$-sets and continuous images. 

A space $X$ is called \textit{$\kappa$-stable}, for an infinite cardinal $\kappa$, if every continuous image $Y$ of $X$ which admits a continuous one-to-one mapping onto a space $Z$ with $w(Z)\leq\kappa$ satisfies $nw(Y)\leq\kappa$. If $X$ is $\kappa$-stable for each $\kappa\geq\omega$, we say that $X$ is \textit{stable}. It is known that every Lindel\"of $\Sigma$-space is stable \cite[Proposition~5.3.15]{AT}.

A space $X$ is \textit{weakly Lindel\"of} if every open cover of $X$ contains a countable subfamily whose union is dense in $X$. Every Lindel\"of space as well as every space of countable cellularity is weakly Lindel\"of.

Let $G$ be a topological group. Given an infinite cardinal $\kappa$, we say that $G$ is \textit{$\kappa$-narrow} if for every neighborhood $U$ of the identity in $G$, there exists a subset $C$ of $G$ with $|C|\leq\kappa$ such that $CU=G$ or, equivalently, $UC=G$. The minimum cardinal $\kappa\geq\om$ such that $G$ is $\kappa$-narrow is denoted by $ib(G)$. Every topological group of countable cellularity is $\om$-narrow \cite[Proposition~5.2.1]{AT}, and the same conclusion holds for weakly Lindel\"of topological groups \cite[Proposition~5.2.8]{AT}.

\section{The weight of Lindel\"of $\Sigma$-spaces}\label{Sec:Con}
By $C(X)$ we denote the family of continuous real-valued functions on a given space $X$.

\begin{thm}\label{Th:1}
The inequalities $w(X)\leq |C(X)|\leq nw(X)^{Nag(X)}$ are valid for every Tychonoff space $X$. 
\end{thm}

\begin{proof}
 Let $\kappa=Nag(X)$. Denote by $C_p(X)$ the set $C(X)$ endowed with the pointwise convergence topology. It follows from \cite[Theorem~I.1.3]{Ar} that $nw(C_p(X))=nw(X)$. Hence $C_p(X)$ contains a dense subset $D$ with $|D|\leq nw(X)$. Let us note that $l(X^n)\leq Nag(X^n)=Nag(X)=\kappa$ for each integer $n\geq 1$. Therefore the tightness of $C_p(X)$ does not exceed $\kappa$ by \cite[Theorem~II.1.1]{Ar}. Further, every continuous image $Y$ of $X$ satisfies $Nag(Y)\leq Nag(X)$. According to \cite[Proposition~5.3.15]{AT} it now follows from $Nag(X)=\kappa$ that the space $X$ is $\kappa$-stable. Hence the space $C_p(X)$ is \textit{$\kappa$-monolithic} by \cite[Theorem~II.6.8]{Ar}, i.e.~the closure of every subset $B$ of $C_p(X)$ with $|B|\leq\kappa$ has a network of cardinality $\leq\kappa$. In particular, the closure of every subset $B$ of $C_p(X)$ with $|B|\leq\kappa$ has cardinality at most $2^\kappa$. 

Summing up, we can write
$$
C_p(X)=\bigcup\{\overline{B}: B\subset D,\ |B|\leq\kappa\}.
$$
Since there are at most $|D|^\kappa$ subsets $B$ of $D$ satisfying $|B|\leq\kappa$ and the closure of each of them is of cardinality $\leq 2^\kappa$, we infer that $|C_p(X)|\leq nw(X)^\kappa\cdot 2^\kappa= nw(X)^\kappa$. Finally, the family of co-zero sets in $X$ forms a base for $X$, so $w(X)\leq |C(X)|\leq nw(X)^\kappa$. 
\end{proof}

\begin{coro}\label{Cor:1}
If $X$ is a Lindel\"of $\Sigma$-space satisfying $nw(X)\leq\cont$, then $|C(X)|\leq\cont$ and $w(X)\leq\cont$.
\end{coro}

The next result is a generalization of Theorem~\ref{Th:1}.

\begin{thm}\label{Th:2}
If $Y$ is a dense subspace of a space $X$, then $|C(X)|\leq nw(Y)^{Nag(Y)}$ and $w(X)\leq nw(Y)^{Nag(Y)}$.
\end{thm}

\begin{proof}
As in the proof of Theorem~\ref{Th:1}, it suffices to verify that $|C(X)|\leq nw(Y)^{Nag(Y)}$. Consider the restriction mapping $r\colon C(X)\to C(Y)$, where $r(f)$ is the restriction of $f\in C(X)$ to the subspace $Y$ of $X$. Since $Y$ is dense in $X$, the mapping $r$ is one-to-one. Hence $|C(X)|\leq |C(Y)|$. To finish the proof it suffices to apply Theorem~\ref{Th:1} to $Y$ in place of $X$.
\end{proof}

\begin{coro}\label{Cor:2}
Suppose that a Tychonoff space $X$ with $nw(X)\leq\cont$ contains a dense Lindel\"of $\Sigma$-subspace. Then $|C(X)|\leq \cont$ and $w(X)\leq\cont$. In particular, the Stone--\v{C}ech compactification $\beta{X}$ of $X$ satisfies $w(\beta{X})\leq\cont$.
\end{coro}

\begin{proof}
Notice that a dense subspace of $X$ is dense in $\beta{X}$. Hence the required conclusions follow from Theorem~\ref{Th:2}.
\end{proof}

\begin{coro}\label{Cor:3}
Suppose that a Tychonoff space $X$ contains a dense Lindel\"of $\Sigma$-subspace $Y$. Then the following are equivalent:
\begin{enumerate}
\item[{\rm (a)}] $Y$ admits a continuous bijection onto a space of weight $\leq\cont;$
\item[{\rm (b)}] $X$ admits a continuous bijection onto a space of weight $\leq\cont;$
\item[{\rm (c)}] $nw(Y)\leq\cont;$
\item[{\rm (d)}] $nw(X)\leq\cont;$
\item[{\rm (e)}] $w(Y)\leq\cont;$
\item[{\rm (f)}] $w(X)\leq\cont$.
\end{enumerate}
\end{coro}

\begin{proof}
The implications (b)$\implies$(a), (d)$\implies$(c), and (f)$\implies$(e) are evident. The validity of the implications (f)$\implies$(d)$\implies$(b) and (e)$\implies$(c)$\implies$(a) is also clear. So it suffices to verify that (a) implies (f). 

Suppose that $Y$ admits a continuous one-to-one mapping onto a Tychonoff space of weight $\leq\cont$, i.e.~$iw(Y)\leq\cont$. Then $nw(Y)\leq Nag(Y)\cdot iw(Y)\leq\cont$ by \cite[Proposition~5.3.15]{AT}. Applying Theorem~\ref{Th:2} we deduce that $w(X)\leq nw(Y)^{Nag(Y)}\leq\cont^{\omega}=\cont$.
\end{proof}

\section{The case of topological groups}\label{Sec:3}
Now we apply Theorems~\ref{Th:1} and~\ref{Th:2} to topological groups. The following lemma is a part of the topological group folklore.

\begin{lemma}\label{Le:We}
If $X$ is a dense subspace of a topological group $G$, then $w(X)=w(G)$.
\end{lemma}

\begin{proof}
It is clear that $w(X)\leq w(G)$, so we verify only that $w(G)\leq w(X)$. According to \cite[Proposition~5.2.3]{AT} we have that $w(G)=ib(G)\cdot \chi(G)$. Let $e$ be the identity element of $G$. Since $G$ is homogeneous, we can assume without loss of generality that $e\in X$. It follows from the regularity of the space $G$ and the density of $X$ in $G$ that $\chi(e,G)=\chi(e,X)$ (see \cite[2.1.C(a)]{Eng}). Hence $\chi(G)\leq w(X)$. Since $X$ is dense in $G$, every open cover of $G$ contains a subfamily of cardinality at most $l(X)$ whose union is dense in $G$, i.e.~$wl(G)\leq l(X)\leq w(X)$. According to \cite[Proposition~5.2.8]{AT}, we see that $ib(G)\leq wl(G)$, so $ib(G)\leq w(X)$. Summing up, $w(G)=ib(G)\cdot \chi(G)\leq w(X)$.
\end{proof}

\begin{thm}\label{Th:3}
Let $G$ be a dense subgroup of a topological group $H$. If $G$ is a Lindel\"of $\Sigma$-group, then $w(H)=w(G)\leq\psi(G)^\om$.
\end{thm}

\begin{proof}
It is clear that the Lindel\"of group $G$ is $\om$-narrow. Applying \cite[Proposition~5.2.11]{AT} we can find a continuous isomorphism (not necessarily a homeomorphism) $f\colon G\to K$ onto a Hausdorff topological group $K$ satisfying $w(K)\leq\kappa$, where $\kappa=\psi(G)$. Since every Lindel\"of $\Sigma$-space is stable (see Proposition~5.3.15 or Corollary~5.6.17 of \cite{AT}), we conclude that $nw(G)\leq\kappa$. Therefore Theorem~\ref{Th:1} implies that $w(G)\leq\kappa^\om$. Hence $w(H)=w(G)\leq\kappa^\om$, by Lemma~\ref{Le:We}.
\end{proof}

The following result is similar in spirit to Theorem~\ref{Th:3}. In it, we weaken the conditions on $G$ by assuming it to be a \textit{subspace} of $H$. The price of this is that the upper bound for the weight of $H$ goes up to $2^\kappa$.

\begin{thm}\label{Th:3b}
Let $X$ be a subspace of a topological group $H$. If $X$ is a Lindel\"of $\Sigma$-space and generates a dense subgroup of $H$, then $w(X)\leq w(H)\leq 2^{\psi(X)}$.
\end{thm}

\begin{proof}
Let $\kappa=\psi(X)$. Denote by $\mathcal{F}$ a countable family of closed sets in the Stone--\v{C}ech compactification $\beta{X}$ of $X$ such that $\mathcal{F}$ separates points $X$ from $\beta{X}\sm X$. For every $x\in X$, let $C(x)=\bigcap\{F\in\mathcal{F}: x\in F\}$. Then $C(x)$ is a compact subset of $X$, for each $x\in X$. Since $|\mathcal{F}|\leq\om$, we see that the family $\mathcal{C}=\{C(x): x\in X\}$ has cardinality at most $\cont$. Every element $C\in\mathcal{C}$ satisfies $\psi(C)\leq\psi(X)=\kappa$, so the compactness of $C$ implies that $\chi(C)=\psi(C)\leq\kappa$. Hence, by Arhangel'skii's theorem, $|C|\leq 2^\kappa$ for each $C\in\mathcal{C}$. Since $X=\bigcup\mathcal{C}$, we see that $|X|\leq |\mathcal{C}|\cdot 2^\kappa = 2^\kappa$. In particular, $nw(X)\leq |X|\leq 2^\kappa$ and the dense subgroup of $H$ generated by $X$, say, $G$ satisfies the same inequality $nw(G)\leq |G|\leq 2^\kappa$. Notice that $G$ is a Lindel\"of $\Sigma$-space, by \cite[Proposition~5.3.10]{AT}. Applying Theorem~\ref{Th:2} and Lemma~\ref{Le:We}, we conclude that $w(X)\leq w(H)=w(G)\leq (2^\kappa)^\om=2^\kappa$. 
\end{proof}

The upper bound on the weight of $H$ in Theorem~\ref{Th:3b} is exact. Indeed, let $\kappa\geq\om$ be a cardinal, $Y=2^\kappa$ the Cantor cube of weight $\kappa$, and $X$ the Alexandroff duplicate of $Y$ (see \cite{Eng1} or \cite{BlT} for more details on the properties of Alexandroff duplicates). Then $\chi(X)=\chi(Y)=\kappa$, while the compact space $X$ contains an open discrete subspace of cardinality $|Y|=2^\kappa$, so $w(X)=2^\kappa$. Denote by $H$ the free topological group over $X$. Then $X$ generates $H$ algebraically and $w(H)\geq w(X)=2^\kappa$. Since $X$ is compact, it is a Lindel\"of $\Sigma$-space. Let us also note that the exact value of the weight of $H$ is $2^\kappa$. 
To see this, we apply the fact that the $\sigma$-compact group $H$ satisfies $nw(H)=nw(X)=w(X)=2^\kappa$ (see \cite[Corollary~7.1.17]{AT}), so Theorem~\ref{Th:1} implies that $w(H)\leq nw(H)^\om=(2^\kappa)^\om=2^\kappa$. Summing up, $w(H)=2^\kappa=2^{\chi(X)}=2^{\psi(X)}$.\smallskip

We do not know, however, whether the inequality $w(H)\leq 2^{\psi(X)}$ in the above theorem can be improved as stated in Theorem~\ref{Th:3}, provided $X$ is dense in $H$:

\begin{problem}\label{Prob:n}
Suppose that a Lindel\"of $\Sigma$-space $X$ is a dense subspace of a topological group $H$. Is it true that $w(H)\leq \psi(X)^\om$? 
\end{problem}

It is easy to see that if $X$ and $H$ are as in Problem~\ref{Prob:n}, then $w(H)=\chi(H)=\chi(X)$. Hence the affirmative answer to the problem would follow if we were able to show that $\chi(x,X)\leq\psi(X)^\om$ for some point $x\in X$.

The next problem is not related directly to the content of this section, but it is close in spirit to Problem~\ref{Prob:n} and is motivated by the famous problem of Arhangel'skii about the cardinality of regular Lindel\"of spaces of countable pseudocharacter.

\begin{problem}\label{Prob:Ar}
Let $X$ be a Lindel\"of space of countable pseudocharacter which is homeomorphic to a dense subspace of a Hausdorff topological group. Is the cardinality of $X$ not greater than $\cont$?
\end{problem}

The requirement on $X$ in the above problem to be a dense subspace of a topological group gives new 
restraints on cardinal characteristics of $X$. For example, such a space $X$ has to satisfy $c(X)\leq\cont$. Indeed, let $G$ be a topological group containing $X$ as a dense subspace. Since $X$ is Lindel\"of, the group $G$ is weakly Lindel\"of and, by \cite[Proposition~5.2.8]{AT}, is $\om$-narrow. According to \cite[Theorem~5.4.10]{AT}, the cellularity of every $\om$-narrow topological group does not exceed $\cont$. As $X$ is dense in $G$, we conclude that $c(X)=c(G)\leq\cont$.\smallskip

Theorems~\ref{Th:3} and~\ref{Th:3b} make it natural to introduce the following definition, with the aim to    extend the two results to wider classes of topological groups.

\begin{Def}\label{Def:1}
Let $G$ be a topological group and $\kappa,\lambda$ infinite cardinals with $\kappa\leq\lambda$. We say that $G$ is \textit{$(\kappa,\lambda)$-moderate} if every continuous homomorphic image $H$ of $G$ with $\psi(H)\leq\kappa$ satisfies $w(H)\leq\lambda$.
\end{Def}

Notice that by Theorem~\ref{Th:3}, every topological group $H$ containing a dense Lindel\"of $\Sigma$-space is $(\kappa,\kappa^\om)$-moderate, for each $\kappa\geq\om$.

In the following proposition we collect a number of well-known results and formulate them in terms of $(\kappa,\lambda)$-moderate groups, as introduced in Definition~\ref{Def:1}. 

\begin{prop}\label{Pro:Mod}
The following are valid for a topological group $G$:
\begin{enumerate}
\item[{\rm (a)}] The group $G$ is $(\kappa,2^{2^\kappa})$-moderate, for each 
                        $\kappa\geq ib(G)$.
\item[{\rm (b)}] If $G$ is compact, then it is $(\kappa,\kappa)$-moderate for each 
                        $\kappa\geq\omega$.
\item[{\rm (c)}] If $G$ is pseudocompact, then it is $(\omega,\omega)$-moderate.
\item[{\rm (d)}] Every Lindel\"of $\Sigma$-group is $(\kappa,\kappa^\om)$-moderate, for each $\kappa\geq\om$.
\end{enumerate}
\end{prop}

\begin{proof}
(a) Let $\kappa\geq ib(G)$ be a cardinal and $f\colon G\to H$ a continuous homomorphism onto a topological group $H$ satisfying $\psi(H)\leq\kappa$. Then $ib(H)\leq ib(G)\leq\kappa$. Hence $|H|\leq 2^{ib(H)\cdot\psi(H)}\leq 2^\kappa$ by \cite[Theorem~5.2.15]{AT}, and $w(H)\leq 2^{|H|}\leq 2^{2^{\kappa}}$. It follows that $G$ is $(\kappa,2^{2^\kappa})$-moderate.

(b) Every compact space $X$ satisfies $\psi(X)=\chi(X)$, while every compact topological group $H$ satisfies $w(H)=\chi(H)$ \cite[Corollary~5.2.7]{AT}. Combining the two equalities, we obtain the required conclusion.

(c) Suppose that $f\colon G\to H$ is a continuous homomorphism of a pseudocompact group $G$ onto a topological group $H$ of countable pseudocharacter. Then $H$ is also pseudocompact. It is well known that every Tychonoff pseudocompact space of countable pseudocharacter has countable character. Hence $H$ is metrizable by the Birkhoff--Kakutani theorem. Finally we note that a pseudocompact metrizable space is compact and second countable. So $w(H)\leq\om$ and therefore the group $G$ is $(\om,\om)$-moderate.

(d) The class of Lindel\"of $\Sigma$-groups is closed under taking continuous homomorphic images, so the required conclusion follows from Theorem~\ref{Th:3}.
\end{proof}

Since every Lindel\"of $\Sigma$-group is $(\om,\cont)$-moderate, the following result generalizes Theorem~\ref{Th:3}.

\begin{thm}\label{Th:31}
Let $G$ be a Lindel\"of $(\om,\cont)$-moderate topological group. Then $w(G)\leq |C(G)|\leq \psi(G)^\omega$, so $G$ is $(\tau,\tau^\om)$-moderate for each $\tau\geq\om$.
\end{thm}

\begin{proof}
Let $\kappa=\psi(G)\geq\omega$. There exists a family $\mathcal{P}$ of open symmetric neighborhoods of the identity element $e$ in $G$ such that $\bigcap\mathcal{P}=\{e\}$ and $|\mathcal{P}|=\kappa$. We can assume without loss of generality that for every $U\in\mathcal{P}$, there exists $V\in\mathcal{P}$ such that $V^3\subset U$. Let us call a sequence $\xi=\{U_n: n\in\omega\}\subset \mathcal{P}$ \textit{admissible} if $U_{n+1}^3 \subset U_n$ for each $n\in\omega$. It is clear that $N_\xi=\bigcap\xi$ is a closed subgroup of type $G_\delta$ in $G$, for each admissible sequence $\xi$. However, the subgroups $N_\xi$ are not necessarily invariant in $G$. Denote by $\mathcal{A}$ the family of subgroups $N_\xi$, where $\xi$ ranges over all admissible sequences in $\mathcal{P}$. Then $|\mathcal{A}|\leq |\mathcal{P}|^\om=\kappa^\omega$.\smallskip

\noindent {\bf Claim.} \textit{Let $\mathcal{N}$ be the family of all invariant subgroups of type $G_\delta$ in $G$. Then the family $\mathcal{N}\cap\mathcal{A}$ is cofinal in $\mathcal{N}$ when the latter is ordered by inverse inclusion.}\smallskip

Let us start the proof of Claim with several simple observations. 

a) First, we note that both families $\mathcal{N}$ and $\mathcal{A}$ are closed under countable intersections. 

b) Second, every neighborhood $U$ of $e$ in $G$ contains an element of $\mathcal{A}$ and an element of $\mathcal{N}$. This is clear for $\mathcal{A}$ since $G$ is Lindel\"of. 
Indeed, if $N\setminus U\neq\emptyset$ for each $N\in\mathcal{A}$, then the property of $\mathcal{A}$ mentioned in a) implies that $(G\setminus U)\cap\bigcap\mathcal{A}\neq \emptyset$, which is impossible since $\bigcap\mathcal{A}=\{e\}$. To find an element $N\in\mathcal{N}$ with $N\subset U$, it suffices to note that the group $G$ is $\om$-narrow and apply \cite[Corollary~3.4.19]{AT}. 

c) Third, every element of $\mathcal{N}$ contains an element of $\mathcal{A}$ and vice versa. To verify this, take an arbitrary element  $N\in\mathcal{N}$. Since $N$ is of type $G_\delta$ in $G$, there exists a sequence $\{U_n: n\in\omega\}$ of open sets in $G$ such that $N=\bigcap_{n\in\omega} U_n$. Making use of b) we find, for every $k\in\om$, an element $N_k\in\mathcal{A}$ such that $N_k\sub U_k$. Then by a), $N^*=\bigcap_{n\in\om} N_k$ is in $\mathcal{A}$ and clearly $N^*\subset N$. Conversely, take an element $N_\xi\in\mathcal{A}$, where $\xi=\{U_k: k\in\omega\}$ is an admissible sequence in $\mathcal{P}$. Applying b) once again we find, for every $k\in\omega$, an element $N_k\in\mathcal{N}$ such that $N_k\subset U_k$. Then $N_*=\bigcap_{k\in\om} N_k$ is in $\mathcal{N}$ and $N_*\subset N_\xi$.

We now turn back to the proof of Claim. Given an arbitrary element $N\in\mathcal{N}$, we have to find an element $\widetilde{N}\in\mathcal{N}\cap\mathcal{A}$ satisfying $\widetilde{N}\subset N$. Let $N_0=N$. Using c) we define a sequence $\{N_k: k\in\om\}$ such that $N_{k+1}\subset N_k$, $N_{2k}\in\mathcal{N}$, and $N_{2k+1}\in\mathcal{A}$ for each $k\in\om$. It follows from a) that $\widetilde{N}=\bigcap_{k\in\om} N_{2k}= \bigcap_{k\in\om} N_{2k+1}$ is in $\mathcal{A}\cap \mathcal{N}$. Since $\widetilde{N} \subset N_0=N$, this completes the proof of Claim.\smallskip

Let $\mathcal{B}=\mathcal{A}\cap\mathcal{N}$. Then $|\mathcal{B}|\leq\mathcal{A}\leq \kappa^\om$. For every $N\in\mathcal{B}$, denote by $\pi_N$ the quotient homomorphism of $G$ onto $G/N$. Since every $N\in\mathcal{B}$ is the intersection of an admissible sequence of neighborhoods of $e$ in $G$, it is easy to verify that the corresponding quotient group $G/N$ has countable pseudocharacter. As $G$ is $(\om,\cont)$-moderate, the weight of the quotient group $G/N$ is at most $\cont$. Clearly the group $G/N$ is Lindel\"of. According to \cite[Theorem~2.2]{CH} this implies that the cardinality of the family of continuous real-valued functions on $G/N$ satisfies $|C(G/N)| \leq\cont^\om=\cont$. For every $N\in\mathcal{B}$, let
\[
C_N(G)=\{g\circ\pi_N: g\in C(G/N)\}.
\]
We claim that $C(G)=\bigcup_{N\in\mathcal{B}} C_N(G)$. Indeed, let $f$ be a continuous real-valued function on $G$. Since every Lindel\"of topological group is $\R$-factor\-iz\-able by \cite[Theorem~8.1.6]{AT}, we can find a continuous homomorphism $p\colon G\to H$ onto a second countable Hausdorff topological group $H$ and a continuous real-valued function $h$ on $H$ such that $f=h\circ p$. Let $K$ be the kernel of the homomorphism $p$. It is clear that $K\in\mathcal{N}$. By our Claim, there exists $N\in\mathcal{B}$ with $N\subset K$. Let $\varphi\colon G/N\to H$ be the natural homomorphism satisfying $p=\varphi\circ\pi_N$. The homomorphism $\varphi$ is continuous since so are $\pi_N$ and $p$, while $\pi_N$ is open. Hence $g=h\circ\varphi$ is a continuous real-valued function on $G/N$ which satisfies $g\circ\pi_N=h\circ\varphi\circ\pi_N=h\circ p=f$. This shows that $f\in C_N(G)$, whence the equality $C(G)=\bigcup_{N\in\mathcal{B}} C_N(G)$ follows. Since $|\mathcal{B}|\leq \kappa^\omega$ and $|C_N(G)|\leq\cont$ for each $N\in\mathcal{B}$, we conclude that $|C(G)|\leq \kappa^\omega$. Thus $w(G)\leq |C(G)|\leq\kappa^\omega$.

Since every continuous homomorphic image of a Lindel\"of $(\om,\cont)$-moderate group is again Lindel\"of and $(\om,\cont)$-moderate, the last assertion of the theorem is immediate from the first one.
\end{proof}

We will see in Example~\ref{Exa:2} that \lq{Lindel\"of\rq} in Theorem~\ref{Th:31} cannot be weakened to \lq{weakly Lindel\"of\rq}, or even replaced with \lq{countably compact\rq}.

The next fact is easily deduced from Theorem~\ref{Th:31}.

\begin{coro}\label{Cor:4}
Let $G$ be a Lindel\"of $\om$-stable topological group satisfying $\psi(G)\leq\cont$. Then $w(G)\leq\cont$.
\end{coro}

\begin{proof}
We claim that the group $G$ is $(\omega,\cont)$-moderate. Indeed, suppose that $f\colon G\to H$ is a continuous homomorphism of $G$ onto a topological group $H$ of countable pseudocharacter. Then $H$ is also Lindel\"of and hence $\om$-narrow. It now follows from \cite[Proposition~5.2.11]{AT} that there exists a continuous isomorphism $i\colon H\to K$ onto a second countable topological group $K$. Since $G$ is $\om$-stable, we conclude that $d(H)\leq nw(H)\leq\om$. Hence $w(H)\leq 2^{d(H)}\leq\cont$. This proves our claim. To complete the argument it suffices to apply Theorem~\ref{Th:31}.
\end{proof}

As usual, we say that $X$ is a \textit{$P$-space} if every $G_\delta$-set in $X$ is open. Since every regular Lindel\"of $P$-space is $\om$-stable \cite[Corollary~5.6.10]{AT}, we have the following:

\begin{coro}\label{Cor:LP}
A Lindel\"of $P$-group $G$ with $\psi(G)\leq\cont$ satisfies $w(G)\leq\cont$.
\end{coro}

It is tempting to conjecture, after Proposition~\ref{Pro:Mod} and Theorem~\ref{Th:31}, that the subgroups of compact topological groups (i.e.~precompact groups) are \lq{moderate\rq} in some sense. For example, it might be a plausible conjecture that every precompact group is $(\om,\cont)$-moderate. We show below that this is not the case and that item (a) of Proposition~\ref{Pro:Mod} is the only restriction on precompact groups in this sense. 

\begin{example}\label{Exa:1}
For every cardinal $\tau\geq\omega$, there exists a precompact Abelian group $G$ satisfying $\psi(G)=\tau$, $d(G)=|G|=2^\tau$, and $w(G)=2^{2^\tau}$. 
\end{example}

\begin{proof}
Let $D$ be a discrete space of cardinality $\tau\geq\om$. Denote by $\beta{D}$ the Stone--\v{C}ech compactification of $D$. Then $|\beta{D}|=2^{2^\tau}$. Consider the space $C_p(\beta{D},\T)$ of continuous functions on $\beta{D}$ with values in the compact circle group $\T$. The subscript \lq{$p$\rq} in $C_p(\beta{D},\T)$ means that this space carries the topology of pointwise convergence on elements of $\beta{D}$, i.e.~$G=C_p(\beta{D},\T)$ is identified with a dense subgroup of the compact topological group $\T^{\beta{D}}$. Hence the topological group $G$ is precompact. Since $G$ is dense in $\T^{\beta{D}}$, Lemma~\ref{Le:We} implies that $w(G)=w(\T^{\beta{D}})= |\beta{D}|= 2^{2^\tau}$.

By \cite[Theorem~I.1.4]{Ar}, we have the equalities
$$
\psi(G)=\psi(C_p(\beta{D},\T))=d(\beta{D})=|D|=\tau.
$$ 
It remains to show that $d(G)=|G|=2^\tau$. According to \cite[Theorem~I.1.5]{Ar}, the density of $C_p(\beta{D},\T)$ is equal to $iw(\beta{D})$, where $iw(\beta{D})$ denotes the minimal cardinal $\lambda\geq\om$ such that $\beta{D}$ admits a continuous one-to-one mapping onto a Tychonoff space of weight $\lambda$. Since $\beta{D}$ is a compact space, it is clear that $iw(\beta{D})= w(\beta{D})=2^\tau$. Therefore, $d(G)=2^\tau$. Finally, by the density of $D$ in $\beta{D}$ we see that $|C_p(\beta{D},\T)|\leq 2^{|D|}=2^\tau$, i.e.~$|G|\leq 2^\tau$. Since $2^\tau=d(G)\leq |G|\leq 2^\tau$, the required equality follows.
\end{proof}

Let us note that the equalities $d(G)=|G|=2^\tau$ in Example~\ref{Exa:1} are not accidental, since every precompact (even $\tau$-narrow) topological group $H$ with $\psi(H)\leq\tau$ admits a continuous isomorphism onto a topological group $K$ of weight $\leq\tau$ and, therefore, $|H|=|K|\leq 2^\tau$.\smallskip

Countably compact groups are pseudocompact and hence $(\om,\om)$-moderate, by (c) of Proposition~\ref{Pro:Mod}. It seems that there are no other restrictions on countably compact groups, except for the obvious one in (a) of the same proposition. The next example confirms this at least in part.   

Let us recall that a space $X$ is called \textit{$\om$-bounded} if the closure of every countable subset of $X$ is compact. All $\om$-bounded spaces are countably compact.

\begin{example}\label{Exa:2}
For every infinite cardinal $\tau$ satisfying $\tau^\om=\tau$, there exists an $\om$-bounded topological Abelian group $G$ such that $\psi(G)=\tau$, $d(G)=|G|=2^\tau$, and $w(G)=2^{2^\tau}$. Hence $G$ fails to be $(\tau,2^\tau)$-moderate.
\end{example}

\begin{proof}
Take an infinite cardinal $\tau$ with $\tau^\om=\tau$. Clearly $\tau\geq\cont$. Let $\Pi=2^I$, where $2=\{0,1\}$ is the two-point discrete group and the index set $I$ satisfies $|I|=2^\tau$. Then $\Pi$, endowed with the usual Tychonoff product topology, is a compact group of density at most $\tau$. Let $S$ be a dense subset of $\Pi$ satisfying $|S|\leq\tau$. By \cite[Corollary~1.2]{CS}, there exists a countably compact subspace (even a subgroup) $X$ of $\Pi$ containing $S$ such that $|X|\leq\tau^\om=\tau$. It is clear that $X$ meets every non-empty $G_\delta$-set in $\Pi$. 

Let $\Pi_\om$ be the \textit{$P$-modification} of the space $\Pi$, i.e.~a base of the topology of $\Pi_\om$ consists of $G_\delta$-sets in $\Pi$. Then $X$ is dense in $\Pi_\om$, so $d(\Pi_\om)\leq\tau$. It is also clear that $w(\Pi_\om)\leq w(\Pi)^\om=(2^\tau)^\om=2^\tau$. As in Example~\ref{Exa:1} we consider the space $C_p(\Pi_\om,\T)$ of continuous functions on $\Pi_\om$ with values in the circle group $\T$. Since $\Pi_\om$ is a regular $P$-space, the subgroup $G= C_p(\Pi_\om,\T)$ of $\T^{\Pi_\om}$ is $\om$-bounded (see \cite[Proposition~2.6]{DTT}). As $G$ is dense in $\T^{\Pi_\om}$, we can apply Lemma~\ref{Le:We} to deduce that $w(G)=w(\T^{\Pi_\om})=|\Pi_\om|=2^{2^\tau}$. 

Let us show that $d(G)=|G|=2^\tau$. Take an arbitrary subset $D$ of $G$ with $|D|<2^\tau$. Every element $f\in D$ is a continuous function on $\Pi_\om$ with values in $\T$. It is clear that the topology of $\Pi_\om$ is the \textit{$\aleph_1$-box} topology of $2^I$ as defined in \cite{CG}. Therefore, we can apply the theorem formulated in the abstract of \cite{CG} (with $\kappa=\aleph_1$ and $\alpha=\cont^+$) to find a subset $J_f$ of the index set $I$ with $|J_f|\leq\cont$ such that $f$ \textit{does not depend on $I\sm J_f$} or, equivalently, $f(x)=f(y)$ for all $x,y\in \Pi_\om$ satisfying $x\hskip-2.4pt\res_{J_f}=y\hskip-2.4pt\res_{J_f}$. Then the subset $J=\bigcup_{f\in D} J_f$ of $I$ satisfies $|J|\leq |D|\cdot\cont<2^\tau$. Take an element $i\in I\sm J$ and points $x,y\in \Pi_\om$ such that $x_i\neq y_i$ and $x_j=y_j$ for each $j\in I$ distinct from $i$. It follows from our definition of $J$ that $f(x)=f(y)$ for all $f\in D$. Therefore, if $f\in G$ and $f(x)\neq f(y)$, then $f\notin\overline{D}$. This proves that the density of $G$ is at least $2^\tau$. Further, since $nw(\Pi_\om)\leq w(\Pi_\om)\leq 2^\tau$, it follows from \cite[Theorem~I.1.3]{Ar} that $nw(G)= nw(\Pi_\om)\leq w(\Pi_\om)\leq 2^\tau$. Hence $d(G)\leq nw(G)\leq 2^\tau$. Combining the two inequalities for $d(G)$, we conclude that $d(G)=2^\tau$. Since $d(\Pi_\om)\leq\tau$, the cardinality of $C_p(\Pi_\om,\T)$ is not greater than $2^\tau$. As in Example~\ref{Exa:1} we deduce the equalities $d(G)=|G|=2^\tau$.

Finally, denote by $r_X$ the restriction mapping of $C_p(\Pi_\om,\T)$ to $C_p(X,\T)$, where $r_X(f)=f\hskip -3.5pt\res_X$ for each $f\in C_p(\Pi_\om,\T)$. Since $X$ is dense in $\Pi_\om$, $r_X$ is a continuous monomorphism. It is clear that $C_p(X,\T)$ is a subspace of $\T^X$, so $\psi(C_p(X,\T))\leq |X|\leq\tau$. As $r_X$ is a continuous monomorphism, we see that $\psi(G)\leq\psi(C_p(X,\T)\leq\tau$.
\end{proof}

\section{Subspaces of separable Hausdorff spaces}\label{Sec:4}
If $X$ is a separable \textit{regular} space, then $w(X)\leq\cont$ by \cite[Theorem~1.5.6]{Eng} and, hence, every subspace $Y$ of $X$ satisfies the same inequality $w(Y)\leq\cont$. However, there exists a separable \textit{Hausdorff} space $Z$ such that $\chi (z_0,Z) = 2^{2^\cont}$ for some point $z_0\in Z$ (see \cite{JK}). We see in particular that $w(Z)= 2^{2^\cont}$. 

It turns out, however, that \lq\lq{good\rq\rq} subspaces of separable Hausdorff spaces have a small weight. It is shown in \cite[Lemma~3.4]{LMT} that every compact subspace of a separable Hausdorff space has weight at most $\cont$. Further, according to \cite[Theorem~3.9]{LMT}, if an \emph{almost connected} pro-Lie group $G$ is homeomorphic with a subspace of a separable Hausdorff space, then $G$ itself is separable and has weight at most $\cont$. (A pro-Lie group $G$ is almost connected if it contains a compact invariant subgroup $N$ such that the quotient group $G/N$ is connected, see \cite{HM2}.) In particular, every connected locally compact group satisfies this conclusion.

Our aim here is to find new classes of spaces and topological groups that behave similarly when embedded in a separable Hausdorff space.

\begin{thm}\label{Th:4}
Let $X$ be a regular Lindel\"of $\Sigma$-space. If $X$ admits a homeomorphic embedding into a separable Hausdorff space, then $w(X)\leq\cont$.
\end{thm}

\begin{proof}
Assume that $X$ is a subspace of a separable Hausdorff space. Since $X$ is Lindel\"of, it follows from \cite[Lemma~3.4]{LMT} that $nw(X)\leq\cont$. Hence Theorem~\ref{Th:1} implies that $w(X)\leq\cont$.
\end{proof}

The original definition of Lindel\"of $\Sigma$-spaces given in \cite{Nag} requires only the Hausdorff separation property. Let us recall that definition. A \textit{Hausdorff} space $X$ is a \textit{Lindel\"of $\Sigma$-space} if there exist families $\mathcal{F}$ and $\mathcal{C}$ of closed subsets of $X$ with the following properties:
\begin{enumerate}
\item[(i)]   $\mathcal{F}$ is countable;
\item[(ii)]  every element of $\mathcal{C}$ is compact and $X=\bigcup\mathcal{C}$;
\item[(iii)] for every $C\in\mathcal{C}$ and every open set $U$ in $X$ with $C\sub U$, there exists $F\in\mathcal{F}$ such that $C\sub F\sub U$. 
\end{enumerate}

In the class of Tychonoff spaces, the above definition of Lindel\"of $\Sigma$-spaces and the definition given in Subsection~\ref{SubS} coincide.

It is now natural to ask whether \lq{regular\rq} can be dropped in Theorem~\ref{Th:4}:

\begin{problem}\label{Prob:LS}
Is it true that every Lindel\"of $\Sigma$-space $X$ homeomorphic to a subspace of a separable Hausdorff space satisfies $w(X)\leq\cont$? 
\end{problem}

Another instance of the phenomenon similar to Theorem~\ref{Th:4} is provided by Lindel\"of $P$-spaces. First we present an auxiliary lemma.

\begin{lemma}\label{Le:aux}
Let $f\colon Y\to Z$ be a continuous mapping of Lindel\"of $P$-spaces. If $Y$ and $Y$ are Hausdorff, then $f$ is closed. Therefore, if $f$ is one-to-one and onto, then it is a homeomorphism.
\end{lemma}

\begin{proof}
First we show that the mapping $f$ is closed. It follows from \cite[Lemma~5.3]{ST} that the spaces $Y$ and $Z$ are zero-dimensional. Let $F$ be a non-empty closed subset of $Y$ and take a point $z\in\overline{f(F)}$. Denote by $\mathcal{N}(z)$ the family of clopen neighborhoods of $z$ in $Z$. Since $Z$ is a zero-dimensional $P$-space, the family $\mathcal{N}(z)$ is closed under countable intersections and $\bigcap\mathcal{N}(z)=\{z\}$. It follows from our choice of $z$ that the family $\{F\cap f^{-1}(V): V\in\mathcal{N}(z)\}$ of non-empty closed sets in $Y$ is closed under countable intersections as well. Since the space $Y$ is Lindel\"of, we conclude that $F\cap f^{-1}(z)\neq\emp$ and hence $z\in f(F)$. This proves that $f$ is a closed mapping. The last assertion of the lemma is evident.
\end{proof}

\begin{thm}\label{Th:5}
If a Lindel\"of $P$-space $X$ is homeomorphic to a subspace of a separable Hausdorff space, then $w(X)\leq\cont$.
\end{thm}

\begin{proof}
Let $Y$ be a separable Hausdorff space containing $X$ as a subspace. Denote by $D$ a countable dense subset of $Y$ and consider the family 
\[
\mathcal{B}=\{\Inte_{Y}\overline{U}: U \mbox{ is open in } Y\}.
\]
Since $U\cap D$ is dense in $U$ for every open set $U$ in $Y$, we see that $|\mathcal{B}|\leq\cont$. It is easy to verify that the family $\mathcal{B}$ constitutes a base for a weaker topology on $Y$, say, $\sigma$. Since the original space $Y$ is Hausdorff, so is $(Y,\sigma)$. Let
$$
\mathcal{C} = \left\{ \bigcap\gamma: \gamma\subset\mathcal{B},\ |\gamma|\leq\om \right\}.
$$
Then $\mathcal{C}$ is a base for a topology $\sigma_\om$ on $Y$ called the \textit{P-modification} of the topology $\sigma$. Notice that $|\mathcal{C}|\leq |\mathcal{B}|^\om \leq\cont$. Since $X$ is a $P$-space, the restriction of $\sigma_\om$ to $X$, say, $\sigma_\om(X)$ is weaker than the original topology of $X$ inherited from $Y$. Hence $X'=(X,\sigma_\om(X))$ is a Lindel\"of $P$-space. It is clear that the space $X'$ is Hausdorff since $\sigma\subset\sigma_\om$.

Let $id_X$ be the identity mapping of $X$ onto $X'$. Then $id_X$ is a continuous bijection of Lindel\"of $P$-spaces, so Lemma~\ref{Le:aux} implies that $id_X$ is a homeomorphism. Since $|\mathcal{C}|\leq\cont$, we conclude that $w(X)=w(X')\leq\cont$.
\end{proof}

The next problem arises in an attempt to generalize both Theorems~\ref{Th:4} and~\ref{Th:5}.

\begin{problem}\label{Prob:1}
Is it true that every regular Lindel\"of subspace of a separable Hausdorff space has weight less than or equal to $\cont$?
\end{problem}

Let us note that every Lindel\"of subspace of a separable Hausdorff space has a network of cardinality $\leq\cont$, by \cite[Lemma~3.4]{LMT}.

It is natural to ask whether a precompact or countably compact topological group $G$ satisfies $w(G)\leq\cont$ or $w(G)\leq 2^\cont$ provided it is homeomorphic to a subspace of a separable Hausdorff space. We answer this question in the negative. This requires a simple lemma.

\begin{lemma}\label{Le:HS}
Let $i\colon Y\to X$ be a continuous bijection of spaces. If $X$ is homeomorphic to a subspace of a separable Hausdorff space, so is $Y$.  
\end{lemma}

\begin{proof}
Let $H$ be a separable Hausdorff space containing $X$ as a subspace. First, we can replace $H$ with the separable space $H^\om$ and consider a copy of $X$ embedded in the first factor $H_{0}=H$, if necessary, thus guaranteeing that $X$ is embedded as a nowhere dense subspace. Let $D$ be a countable dense subset of $H$. Since $X$ is nowhere dense in $H$, the complement $D\setminus \overline{X}$ is also dense in $H$. Hence we can additionally assume that $D\cap X=\emptyset$.

Clearly $K=X\cup D$ is a dense subspace of $H$. We define a mapping $f$ of $L=Y\cup D$ onto $K$ by letting $f(y)=i(y)$ if $y\in Y$ and $f(d)=d$ if $d\in D$ (again we assume that $Y\cap D=\emp$). Then $f$ is a bijection. Let $\sigma$ be the coarsest topology on $L$ satisfying the following two conditions:
\begin{enumerate}
\item[(i)]   the mapping $f\colon (L,\sigma)\to K$ is continuous;
\item[(ii)]  if $U$ is open in $Y$, then $U\cup D$ is open in $(L,\sigma)$.
\end{enumerate}
Since $f$ is a bijection of $L$ onto $K$, it follows from (i) that the space $(L,\sigma)$ is Hausdorff, while (i) and (ii) together imply that the topology of $Y$ inherited from $(L,\sigma)$ is the original topology of $Y$. It is also easy to see that $D$ is dense in $L^*=(L,\sigma)$, i.e.~$L^*$ is separable. Indeed, our definition of $\sigma$ implies that the sets of the form $O=(U\cup D)\cap f^{-1}(V)$, with $U$ open in $Y$ and $V$ open in $K$, form a base for $L^*$. Suppose that $V\neq\emp$. Since the restriction of $f$ to $D$ is the identity mapping of $D$, it follows that $O\cap D=f^{-1}(V)\cap D=V\cap D\neq\emp$. This proves that $D$ is dense in $L^*$. Therefore, $Y$ is a subspace of the separable Hausdorff space $L^*$.
\end{proof}

\begin{prop}\label{Pro:Emb2}
There exists an $\om$-bounded (hence countably compact) topological Abelian group $G$ homeomorphic to a subspace of a separable Hausdorff space and satisfying $d(G)=2^\cont$ and $w(G)=2^{2^\cont}$.
\end{prop}

\begin{proof}
According to Example~\ref{Exa:2} with $\tau=\cont$, there exists an $\om$-bounded topological Abelian group $G$ satisfying $\psi(G)=\cont$, $d(G)=2^\cont$, and $w(G)=2^{2^\cont}$. It is clear that $G$ is precompact and hence $\om$-narrow. By \cite[Proposition~5.2.11]{AT}, we can find a continuous isomorphism $f\colon G\to H$ onto a topological group $H$ with $w(H)\leq \psi(G)=\cont$. The group $H$ is precompact and Abelian. Let $K$ be the completion of $H$. Then the group $K$ is compact and, by Lemma~\ref{Le:We}, it satisfies $w(K)=w(H)\leq\cont$. Applying \cite[Corollary~5.2.7(c)]{AT} we deduce that the group $K$ is separable. Thus $H$ is a subspace of a separable Hausdorff (in fact, normal) space. By Lemma~\ref{Le:HS}, $G$ is homeomorphic to a subspace of a separable Hausdorff space.
\end{proof}


\end{document}